\documentclass[11pt]{article}

\usepackage{graphicx}
\usepackage{amssymb}
\usepackage{enumerate}
\usepackage{amsfonts,mathrsfs}
\usepackage{amsmath}
\usepackage{amsthm}
\usepackage{hyperref}
\usepackage{geometry}
\usepackage{color}
\usepackage{caption}
\usepackage{subcaption}

\newcommand{\R}{\mathbb{R}}

\newcommand{\N}{\mathbb{N}}

\newcommand{\g}{\gamma}

\providecommand{\keywords}[1]{\textit{\textit{Key words:}} #1}
\providecommand{\subjclass}[1]{\textit{\textit{2010 Mathematics Subject Classification:}} #1}

\newtheorem{theorem}{Theorem}
\newtheorem{proposition}{Proposition}
\newtheorem{lemma}{Lemma}
\theoremstyle{remark}
\newtheorem{remark}{Remark}
\theoremstyle{definition}
\newtheorem{definition}{Definition}
\theoremstyle{corollary}
\newtheorem{corollary}{Corollary}

\begin{document}

\title{\bf Spectral stability of periodic waves \\in the generalized reduced Ostrovsky equation}

\author{Anna Geyer$^{1}$ and Dmitry E. Pelinovsky$^{2}$ \vspace{1em} \\
{\small $^{1}$ Faculty of Mathematics, University of Vienna, Oskar-Morgenstern-Platz 1, 1090 Vienna, Austria} \\
{\small E-mail: anna.geyer@univie.ac.at} \\
{\small $^{2}$ Department of Mathematics, McMaster University, Hamilton, Ontario, Canada, L8S 4K1}  \\
{\small E-mail: dmpeli@math.mcmaster.ca} }

\date{\today}
\maketitle

\begin{abstract}
We consider stability of periodic travelling waves in the generalized reduced Ostrovsky equation
with respect to co-periodic perturbations. Compared to the recent literature,
we give a simple argument that proves spectral stability of all smooth periodic travelling waves
independent of the nonlinearity power. The argument is based on the energy convexity
and does not use coordinate transformations of the reduced Ostrovsky equations to the semi-linear
equations of the Klein--Gordon type.
\end{abstract}

\keywords{reduced Ostrovsky equations; stability of periodic waves; energy-to-period map; negative index theory}

\subjclass{35B35, 35G30}

\section{Introduction}
\label{sec:intro}

We address the generalized reduced Ostrovsky equation written in the form
\begin{equation}
\label{redOst}
(u_t + u^p u_x)_x = u,
\end{equation}
where $p \in \mathbb{N}$ is the nonlinearity power and $u$ is a real-valued function of $(x,t)$.
This equation was derived in the context of long surface and internal gravity waves in a rotating fluid
for $p = 1$ \cite{Ostrov} and $p = 2$ \cite{Grimshaw}. These two cases are the only cases, for which
the reduced Ostrovsky equation is transformed to integrable semi-linear equations of the Klein--Gordon type
by means of a change of coordinates \cite{JG,GH}.

We consider existence and stability of travelling periodic waves in
the generalized reduced Ostrovsky equation (\ref{redOst}) for any $p \in \mathbb{N}$.
The travelling $2T$-periodic waves are given by
$u(x,t) = U(x-ct)$, where $c > 0$ is the wave speed,
$U$ is the wave profile satisfying the boundary-value problem
\begin{equation}
\label{E ODE}
\frac{d}{dz} \left[ (c - U^p) \frac{dU}{dz} \right] + U(z) = 0, \quad U(-T) = U(T), \quad U'(-T) = U'(T),
\end{equation}
and $z = x-ct$ is the travelling wave coordinate. We are looking
for smooth periodic waves $U \in H^{\infty}_{\rm per}(-T,T)$ satisfying (\ref{E ODE}).
It is straightforward to check that periodic solutions of the second-order
equation (\ref{E ODE}) correspond to  level curves of the first-order invariant,
\begin{equation}
\label{first-order}
E = \frac{1}{2} (c - U^p)^2 \left( \frac{dU}{dz} \right)^2 + \frac{c}{2} U^2 - \frac{1}{p+2} U^{p+2} = {\rm const}.
\end{equation}

We add a {\em co-periodic} perturbation to the travelling wave, that is,
a perturbation with the same period $2T$.
Separating the variables, the spectral stability problem for the perturbation $v$ to $U$
is given by $\lambda v = \partial_z L v$, where
\begin{equation}
\label{operator-L}
L = P_0 \left( \partial_z^{-2} + c - U(z)^p \right) P_0 : \; \dot{L}^2_{\rm per}(-T,T) \to \dot{L}^2_{\rm per}(-T,T),
\end{equation}
where $\dot{L}^2_{\rm per}(-T,T)$ denotes the space of $2T$-periodic, square-integrable functions
with zero mean and $P_0 : L^2_{\rm per}(-T,T) \to \dot{L}^2_{\rm per}(-T,T)$ is the projection operator
that removes the mean value of $2T$-periodic functions.

\begin{definition}
\label{def-stability}
We say that the travelling wave is spectrally stable
with respect to co-periodic perturbations if the spectral problem
$\lambda v = \partial_z L v$ with $v \in \dot{H}^1_{\rm per}(-T,T)$
has no eigenvalues $\lambda \notin i \mathbb{R}$.
\end{definition}

Local solutions of the Cauchy problem associated with the generalized reduced Ostrovsky equation (\ref{redOst})
exist in the space $\dot{H}^s_{\rm per}(-T,T)$ for $s > \frac{3}{2}$ \cite{SSK10}.
For sufficiently large initial data, the local solutions break in finite time,
similar to the inviscid Burgers equation \cite{LPS1,LPS2}. However, if the initial data $u_0$ is small in a suitable norm,
then local solutions are continued for all times in the same space, at least in the integrable
cases $p = 1$ \cite{GP} and $p = 2$ \cite{PelSak}.

Travelling periodic waves to the generalized reduced Ostrovsky equation (\ref{redOst})
were recently considered in the cases $p = 1$ and $p = 2$.
In these cases, travelling waves can be found in the explicit form given by the Jacobi elliptic functions
after a change of coordinates \cite{JG,GH}. Exploring this idea further,
it was shown in  \cite{Hakkaev1,Hakkaev2,Stefanov} that the spectral stability of travelling periodic waves can be
studied with the help of the eigenvalue problem $M \psi = \lambda \partial_z \psi$, where $M$ is a second-order Schr\"{o}dinger
operator.
Independently, by using higher-order conserved quantities
which exist in the integrable cases $p = 1$ and $p = 2$,
it was shown in \cite{JP} that the travelling periodic waves are unconstrained
minimizers of energy functions in suitable function spaces with respect to {\em subharmonic} perturbations,
that is, perturbations with a multiple period to the periodic waves. This result yields not only spectral but
also nonlinear stability of the travelling wave. The nonlinear stability of periodic waves
was established analytically for small-amplitude waves and shown numerically for waves of arbitrary amplitude \cite{JP}.

In this paper, we give a simple argument that proves spectral stability of all smooth periodic travelling waves
to the generalized reduced Ostrovsky equation (\ref{redOst}) independently of the nonlinearity power $p$ and
the wave amplitude.
The spectral stability of periodic waves is defined here with respect
to co-periodic perturbations in the sense of Definition \ref{def-stability}. The argument is based on convexity of the energy function
\begin{equation}
\label{energy}
H(u) = -\frac{1}{2} \| \partial_x^{-1} u \|_{L^2_{\rm per}}^2 - \frac{1}{(p+1)(p+2)} \int_{-T}^T u^{p+2} dx,
\end{equation}
at the travelling wave profile $U$ in the energy space with fixed momentum,
\begin{equation}
\label{momentum-space}
X_q = \left\{ u \in \dot{L}^2_{\rm per}(-T,T) \cap L^{p+2}_{\rm per}(-T,T) : \quad \| u \|^2_{L^2_{\rm per}} = 2 q > 0 \right\}.
\end{equation}
Note that the self-adjoint operator $L$ given by (\ref{operator-L}) is the Hessian operator of
the extended energy function $F(u)=H(u) + c Q(u)$, where
\begin{equation}
\label{momentum}
Q(u) = \frac{1}{2} \| u \|^2_{L^2_{\rm per}}
\end{equation}
is the momentum function. The energy $H(u)$ and momentum $Q(u)$,
and therefore the extended energy $F(u)$, are constants of motion,
as can be seen readily by writing the evolution equation \eqref{redOst} in Hamiltonian form
as $u_t = \partial_x  {\rm grad} H(u)$.  Notice that the traveling wave profile $U$
is a critical point of the extended energy function $F(u)$ in the sense that
the Euler--Lagrange equations for $F(u)$ are identical to the boundary-value problem
\eqref{E ODE} after the second-order equation is integrated twice with zero mean.

The outline of the paper is as follows.
Adopting the approach from \cite{Vill1,Vill2,Grau2011}, we prove in Section \ref{section-monoton} that the energy-to-period map
$E \mapsto 2T$ is strictly monotonically decreasing for the family of smooth periodic solutions
satisfying (\ref{E ODE}) and (\ref{first-order}). This result holds for every fixed $c > 0$.
Thanks to monotonicity of the energy-to-period map $E \mapsto 2T$, the inverse mapping defines
the first-order invariant $E$ in terms of the half-period $T$ and the speed $c$. We denote
this inverse mapping by $E(T,c)$.

In Section \ref{section-speed}, we consider continuations of the family of smooth periodic solutions
with respect to parameter $c$ for every fixed $T > 0$ and prove that $E(T,c)$ is an increasing function
of $c$ within a nonempty interval $(c_0(T),c_1(T))$, where $0 < c_0(T) < c_1(T) < \infty$.
We also prove that the momentum $Q(u)$ evaluated at $u = U$ is an increasing function of $c$
for every fixed $T > 0$.

In Section \ref{section-negative}, we use the monotonicity of the mapping $E \mapsto 2T$ for every fixed $c > 0$
and prove that the self-adjoint operator $L$ given by (\ref{operator-L}) has a simple negative eigenvalue,
a one-dimensional kernel, and the rest of its spectrum is bounded from below by a positive number.

Finally, in Section \ref{section-Krein}, we prove that the operator $L$ constrained on the space
\begin{equation}
\label{constraint}
L^2_c = \left\{ u \in \dot{L}^2_{\rm per}(-T,T) : \quad \langle U, u \rangle_{L^2_{\rm per}} = 0 \right\}
\end{equation}
is strictly positive except for the one-dimensional kernel induced by the translational symmetry.
This gives convexity of $H(u)$ at $u = U$ in space of fixed $Q(u)$ given by (\ref{momentum-space}).
By using the standard Hamilton--Krein theorem in \cite{Haragus} (see also the reviews in \cite{Kollar2014}
and \cite{Pel}), this rules out existence of eigenvalues $\lambda \notin i \mathbb{R}$
of the spectral problem $\lambda v = \partial_z L v$ with $v \in \dot{H}^1_{\rm per}(-T,T)$.

All together, the existence and spectral stability of smooth periodic travelling waves
of the generalized reduced Ostrovsky equation (\ref{redOst}) is summarized in the following theorem.

\vspace{2cm}

\begin{theorem}
\label{theorem-main}
For every $c > 0$ and $p \in \mathbb{N}$,
\begin{itemize}
\item[(a)]  there exists a smooth family of periodic solutions $U \in \dot{L}^2_{\rm per}(-T,T) \cap
H^{\infty}_{\rm per}(-T,T)$  of equation (\ref{E ODE}),
parameterized by the energy $E$ given in  (\ref{first-order}) for $E \in (0,E_c)$, with
\begin{equation*}
E_c = \frac{p}{2(p+2)} c^{\frac{p+2}{p}},
\end{equation*}
such that the energy-to-period map $E \mapsto 2T$ is smooth and strictly monotonically decreasing.
Moreover, there exists $T_1 \in (0,\pi)$ such that
\begin{equation*}
T \to \pi c^{\frac{1}{2}} \quad \mbox{\rm as} \quad E \to 0 \quad \mbox{\rm and} \quad T \to T_1 c^{\frac{1}{2}} \quad \mbox{\rm as} \quad E \to E_c;
\end{equation*}

\item[(b)] for each point $U$ of the family of periodic solutions, the operator $L$ given by (\ref{operator-L})
has a simple negative eigenvalue, a simple zero eigenvalue associated with ${\rm Ker}(L) = {\rm span}\{\partial_z U \}$,
and the rest of the spectrum is positive and  bounded away from zero;

\item[(c)] the spectral problem $\lambda v = \partial_z Lv$ with $v \in \dot{H}^1_{\rm per}(-T,T)$
admits no eigenvalues $\lambda \notin i \mathbb{R}$.
\end{itemize}
Consequently, periodic waves of the generalized reduced Ostrovsky equation (\ref{redOst})
are spectrally stable with respect to co-periodic perturbations in the sense of Definition \ref{def-stability}.
\end{theorem}

We now compare our result  to the existing literature on spectral and orbital stability of
periodic waves with respect to co-periodic perturbations.
First, in comparison with the analysis in \cite{Hakkaev2}, the result of Theorem \ref{theorem-main}
is more general since $p \in \mathbb{N}$ is not restricted to the integrable cases $p = 1$ and $p = 2$.
On a technical level, the method of proof of Theorem \ref{theorem-main} is simple and robust,
so that many unnecessary explicit computations from \cite{Hakkaev2} are avoided. Indeed, in
the transformation of the spectral problem $\lambda v = \partial_z L v$
to the spectral problem $M \psi = \lambda \partial_z \psi$, where $M$ is a second-order Schr\"{o}dinger
operator from $H^2_{\rm per}(-T,T) \to L^2_{\rm per}(-T,T)$, the zero-mean constraint is lost\footnote{Note
that this transformation reflects the change of coordinates owing to which the reduced Ostrovsky equations
are reduced to the semi-linear equations of the Klein--Gordon type. This transformation also changes the period
of the travelling periodic wave.}. Consequently, the operator $M$ was found in \cite{Hakkaev2}
to admit two negative eigenvalues in $L^2_{\rm per}(-T,T)$, which are computed explicitly
by using eigenvalues of the Schr\"{o}dinger operator with elliptic potentials.
By adding three constraints for the spectral problem $M \psi = \lambda \partial_z \psi$, the authors of \cite{Hakkaev2}
showed that the operator $M$ becomes positive on the constrained space, again by means
of symbolic computations involving explicit Jacobi elliptic functions.
All these technical details become redundant in our simple approach.

Second, we mention another type of improvement of our method compared to the analysis of spectral stability of
periodic waves in other nonlinear evolution equations \cite{Nataly1,Nataly2}. By establishing
first the monotonicity of the energy-to-period map $E \mapsto 2T$ for a smooth family of periodic waves,
we give a very precise count on the number of negative eigenvalues of the operator $L$ in $L^2_{\rm per}(-T,T)$
without doing numerical approximations on solutions of the homogeneous equation $L v = 0$.
Indeed, the smooth family of periodic waves has a limit to zero solution, for which
eigenvalues of $L$ in $L^2_{\rm per}(-T,T)$ are found from Fourier series. The zero eigenvalue of
$L$ is double in this limit and it splits once the amplitude of the periodic wave
becomes nonzero. Owing to the monotonicity of the map $E \mapsto 2T$ and continuation arguments,
the negative index of the operator $L$ remains invariant
along the entire family of the smooth periodic waves. Therefore, the negative
index of the operator $L$ is found for the entire family of periodic waves  by a simple argument,
again avoiding cumbersome analytical or approximate numerical computations.

Finally, we also mention that the spectral problem $\lambda v = \partial_z L v$ is typically difficult
when it is posed in the space $L^2_{\rm per}(-T,T)$ because the mean-zero constraint is needed on $v$
in addition to the orthogonality condition $\langle U, v \rangle_{L^2_{\rm per}} = 0$. The two constraints
are taken into account by studying the two-parameter family of smooth periodic waves
and working with a $2$-by-$2$ matrix of projections \cite{Bronski,Johnson}. This complication is avoided
for the reduced Ostrovsky equation (\ref{redOst}) because the spectral problem
$\lambda v = \partial_z L v$ is posed in space $\dot{L}^2_{\rm per}(-T,T)$ and the only
orthogonality condition $\langle U, v \rangle_{L^2_{\rm per}} = 0$ is studied with
the help of identities satisfies by the periodic wave $U$.

As a limitation of the results of Theorem \ref{theorem-main} we mention that the nonlinear orbital
stability of travelling periodic waves cannot be established for the reduced Ostrovsky equations (\ref{redOst})
by using the energy function (\ref{energy}) in space (\ref{momentum-space}).
This is because the local solution is defined in
$\dot{H}^s_{\rm per}(-T,T)$ for $s > \frac{3}{2}$ \cite{SSK10}, whereas the energy function
is defined in $\dot{L}^2_{\rm per}(-T,T) \cap L^{p+2}_{\rm per}(-T,T)$. As a result, coercivity
of $H(u)$ in the space of fixed momentum (\ref{momentum-space}) only controls the $L^2$ norm of
time-dependent perturbations. Local well-posedness in such spaces of low regularity
is questionable and so is the proof of orbital stability of the travelling periodic waves
in the time evolution of the reduced Ostrovsky equations (\ref{redOst}).

\section{Monotonicity of the energy-to-period map}
\label{section-monoton}

Traveling wave solutions of  the reduced Ostrovsky equation~\eqref{redOst} are  solutions of the
second-order differential equation~\eqref{E ODE} with fixed $c > 0$ and $p \in \mathbb{N}$.
The following lemma establishes a correspondence between the smooth periodic solutions of
the second-order equation \eqref{E ODE} and the periodic orbits around the center of an
associated planar system, see Figure \ref{Fig phaseportrait}. For lighter notations,
we replace $U(z)$ by $u(z)$ and denote the derivatives in $z$ by primes.

\begin{figure}[h!]
    \centering
        \includegraphics[width=0.35\textwidth]{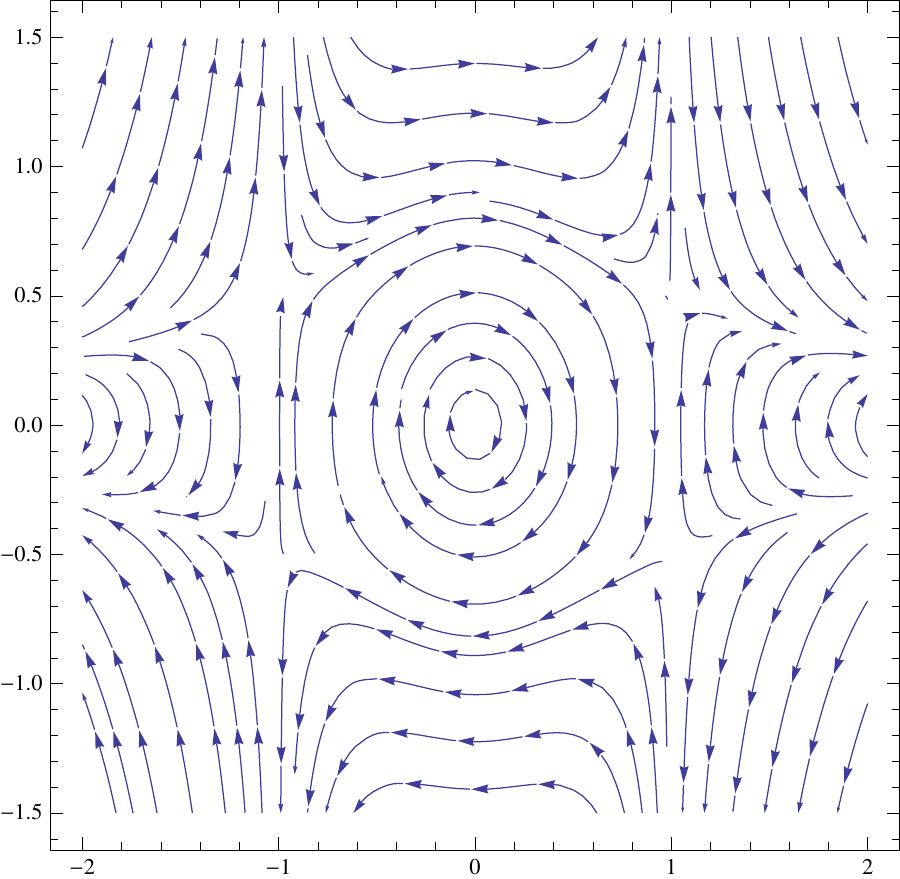}\quad
        \includegraphics[width=0.35\textwidth]{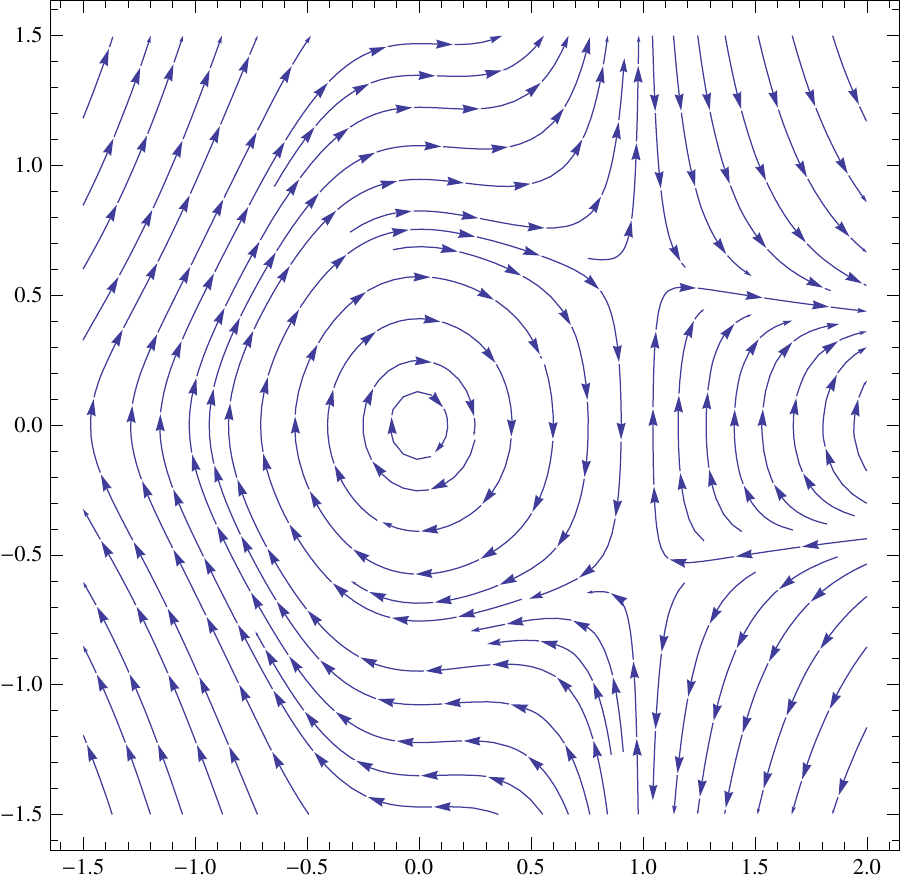}
    \caption{Phase portraits of system \eqref{E sys uv} for $p = 2$ (left) and $p = 1$ (right).}
    \label{Fig phaseportrait}
\end{figure}

\begin{lemma}
\label{L ode sys}
For every $c>0$ and $p\in\N$ the following holds:
\begin{enumerate}[$(i)$]
\item A function $u$ is a smooth periodic solution of equation \eqref{E ODE} if and only if $(u,v)=(u,u')$ is a periodic orbit of the planar differential system
\begin{equation}
    \label{E sys uv}
       \left\{
      \begin{array}{l}
        u' =v,\\[2pt]
        v' =\dfrac{-u + pu^{p-1}v^2}{c-u^p}.
      \end{array}\right.
  \end{equation}
\item The system (\ref{E sys uv}) has a first integral given by \eqref{first-order}, which we write  as
\begin{equation}
\label{E FirstInt}
    E(u,v)=A(u)+B(u) v^2,
\end{equation}
with $A(u) = \frac{c}{2}u^2- \frac{1}{p+2} u^{p+2}$ and $B(u) = \frac{1}{2}(c-u^p)^2$.
\item Every periodic orbit of system \eqref{E sys uv} belongs to the period annulus\footnote{The largest punctured neighbourhood of  a  center which consists entirely of periodic orbits  is called \emph{period annulus}, see~\cite{Chicone}.
} of the center at the origin of the $(u,v)$-plane and lies  inside  some energy level curve of $E$, with $E  \in (0,E_c)$ where
\begin{equation}
\label{E-c}
    E_c := A(c^{1/p}) = \frac{p}{2(p+2)} c^{\frac{p+2}{p}}.
\end{equation}.
 \end{enumerate}
\end{lemma}

\begin{proof}
The assertion in $(ii)$ is proved with a straightforward calculation. To prove $(iii)$,
we notice that system \eqref{E sys uv} has no limit cycles in view of
the existence of a  first integral,  and hence the periodic orbits form period annuli.
A periodic orbit must surround at least one critical point.
The unique critical point of system \eqref{E sys uv} is a center at the origin
on the $(u,v)$ plane,
corresponding to the  energy level $E =0$. In view of the presence of the singular line
$$
\{ u =  c^{1/p}, \quad v \in \mathbb{R}\} \subset \mathbb{R}^2
$$
we may  conclude, applying the Poincar\'{e}-Bendixon Theorem, that the set of periodic orbits
forms a punctured neighbourhood of the center, and that no other period annulus is possible.

It remains to show $(i)$. It is clear that $z \mapsto (u,v)=(u,u')$ is a smooth solution of
the differential system~\eqref{E sys uv}  if and only if $u$ is a smooth solution of
the second-order equation~\eqref{E ODE} satisfying $c\neq u(z)^p$ for all $z$.
We claim that $c\neq u(z)^p$ for all~$z\in\R$ for smooth periodic solutions $u$.
Indeed, let $p$ be odd for simplicity and recall that every periodic orbit
in a planar system has exactly two turning points $(u,u')=(u_{\pm},0)$ per
fundamental period. The turning points correspond to the maximum and minimum
of the periodic solution $u$ and satisfy the equation $A(u_{\pm})=E$.
The graph of $A(u)$ on $\R^+$ has a global maximum at $u=c^{1/p}$ with $E_c$ given in \eqref{E-c}.

The equation $A(u)=E$ has exactly two positive solutions for $E\in(0,E_c)$,
where $u=u_+$ corresponds to the smaller one inside the period annulus.
At $E=E_c$, the equation $A(u) = E$ has only one positive solution
given by $u_+=c^{1/p}$. Now assume that  for a smooth periodic solution $u$,
there exists $z_1$ such that $u(z_1)=c^{1/p}$. Then, equation (\ref{E ODE})
implies that $u'(z_1)=\pm p^{-1/2} c^{-\frac{p-2}{2p}}$, hence
the solution $(u,u')(z)$ to system \eqref{E sys uv} tends to the points
$p_{\pm} = (c^{1/p},\pm p^{-1/2} c^{-\frac{p-2}{2p}})$ as $z\to z_1$.
Since $E(p_{\pm})=E_c$ and by continuity of the first integral, this orbit
lies inside the $E_c$-level set. For such an orbit, we have seen that
its turning point is located at $u_+=c^{1/p}=u(z_1)$.  However, since
$u'(z_1)\neq 0$, this cannot be a turning point, which leads to a contradiction.
Hence, the assertion $(i)$ is proved.
\end{proof}

\medskip

\begin{remark}
\label{R period}
By Lemma \ref{L ode sys}, every smooth periodic solution $u$ of the differential equation \eqref{E ODE} corresponds to  a periodic orbit     $(u,v)=(u,u')$ inside the period annulus  of the differential system  \eqref{E sys uv}. Since $E$ is a first integral of \eqref{E sys uv}, this orbit lies  inside  some energy level curve of $E$, where $E  \in (0,E_c)$. We denote this orbit  by $\gamma_E$.
The period of this orbit is given by
\begin{equation}
\label{E period function}
    2T(E) = \int_{\g_E}\frac{du}{v},
\end{equation}
since  $\frac{du}{dz} = v$ in view of \eqref{E sys uv}.
The energy levels of the first integral  $E$
parameterize the set of periodic orbits inside the period annulus, and therefore this set forms
a smooth family $\{\g_E\}_{E\in (0,E_c)}$. In view of Lemma \ref{L ode sys}, we can therefore
assert that the set of smooth periodic solutions of \eqref{E ODE} forms a smooth family $\{u_E\}_{E\in (0,E_c)}$,
which is parameterized by $E$ as well. Moreover, it ensures that the period $2 T(E)$ of the periodic orbit $\g_E$
is equal to the period of the corresponding smooth periodic solution $u_E$ of the second-order equation \eqref{E ODE}.
\end{remark}

The main result of this section is  the following proposition, from which we conclude that the energy-to-period map
$E \mapsto 2T(E)$ for the smooth periodic solutions  of  equation \eqref{E ODE} is smooth and strictly monotonically decreasing. Together with the above Remark \ref{R period}, this proves statement (a) of Theorem~\ref{theorem-main}.

\begin{proposition}
\label{P T'<0}
For every $c > 0$ and $p \in \mathbb{N}$ the  function
\begin{equation*}
    T: (0,E_c) \longrightarrow  \R^+, \quad E \longmapsto T(E) = \frac{1}{2} \int_{\g_E}\frac{du}{v},
\end{equation*}
is strictly monotonically decreasing and
satisfies
\begin{align}
    T'(E) = -\frac{p}{4 (2+p) E} \int_{\g_E} \frac{u^p}{(c-u^p)}\frac{du}{v} < 0.
\end{align}
\end{proposition}

\begin{proof}
Since $A(u) + B(u) v^2=E$ is constant along an orbit $\g_E$, we find that
\begin{equation}
\label{E ET}
    2 E \,T(E)= \int_{\g_E} B(u) v du + \int_{\g_E} A(u) \frac{du}{v}.
\end{equation}
To compute the derivative of $T$ with respect to $E$, we first resolve the singularity in
the second integral in equation \eqref{E ET}. To this end, recall that the orbit $\g_E$
belongs to the level curve $\{A(u) + B(u) v^2=E\}$ and therefore
 \begin{equation}
 \label{E dvdu}
     \frac{dv}{du} = -\frac{A'(u) + B'(u)v^2}{2B(u)v}
 \end{equation}
along the orbit. Note that $B(u)$ is different from zero for $E \in (0,E_c)$.
Furthermore, $BA/A'$ is bounded on $\gamma_E$. Using the fact that
the integral of a total differential $d$ over the closed orbit $\g_E$ yields zero, we find that
 \begin{eqnarray*}
     0 &=& \int_{\g_E} d \left[ \left(\frac{2BA}{A'}\right)(u) \,v \right ]\\
       &=& \int_{\g_E}\left (\frac{2BA}{A'}\right)'(u) \,v \,du
           + \left (\frac{2BA}{A'}\right)(u)\, dv\\
       &=& \int_{\g_E} \left (\frac{2BA}{A'}\right)'(u) \,v \,du
           - \left (\frac{2BA}{A'} \frac{A'}{2B}\right )(u) \frac{du}{v}
           - \left (\frac{2BA}{A'} \frac{B'}{2B}\right )(u) \,v \,du\\
       &=& \int_{\g_E} \left
       [\left (\frac{2BA}{A'}\right)'(u) -\left (\frac{AB'}{A'}\right )(u)\right ] v\, du
          - A(u)\frac{du}{v},
 \end{eqnarray*}
where we have used relation \eqref{E dvdu} in the third equality.  Denoting
\begin{equation}
\label{E G}
    G = \left (\frac{2BA}{A'} \right )' - \frac{AB'}{A'},
\end{equation}
this ensures that
\begin{equation}
\label{improved-T-E}
2 E T(E) =  \int_{\g_E} \left[ B(u) + G(u) \right] v du,
\end{equation}
where the integrand is no longer singular at the turning points, where the orbit $\gamma_E$ intersects
with the horizontal axis $v = 0$\footnote{The idea for this approach
of resolving the singularity is taken from \cite[Lemma 4.1]{Grau2011},
where the authors prove a more general result for polynomial systems having first integrals of the form \eqref{E FirstInt}.}.
Taking now the derivative of equation \eqref{improved-T-E} with respect to $E$ we obtain  that
\begin{align}
\label{computation-T-E}
   2 T(E) + 2 E \,T'(E) = \int_{\g_E} \frac{B(u) + G(u)}{2B(u)v} du,
\end{align}
where we have used that
$$
\frac{\partial v}{\partial E} = \frac{1}{2 B(u) v}
$$
in view of \eqref{E FirstInt}\footnote{
Note that (\ref{computation-T-E}) also follows by applying
Gelfand-Leray derivatives  in \eqref{improved-T-E}, see \cite{Ilyashenko2006}~Theorem 26.32,~p.~526.}.
From (\ref{computation-T-E}), we conclude that
\begin{eqnarray*}
    2 T'(E) &=& \frac{1}{E} \int_{\g_E} \left (\frac{B+G}{2B}\right ) (u) \frac{du}{v}
             - \frac{1}{E} \int_{\g_E} \frac{du}{v}  \\
           &=&  \frac{1}{E} \int_{\g_E}
           \frac{1}{2B} \left(\left (\frac{2AB}{A'}\right )' - \frac{(AB)'}{A'}\right ) (u) \frac{du}{v}.
\end{eqnarray*}
In view of the expressions for $A$ and $B$ defined in Lemma \ref{L ode sys},
further calculations show that
\begin{align}
\label{T-prime}
    T'(E) = -\frac{p}{4 (2+p) E} \int_{\g_E} \frac{u^p}{(c-u^p)}\frac{du}{v}.
\end{align}
We now need to show that $T'(E) < 0$ for every $E \in (0,E_c)$.
In view of the symmetry of the vector field with respect to the horizontal axis
and taking into account \eqref{E FirstInt}, we write (\ref{T-prime}) in the form
\begin{eqnarray}
\nonumber
    T'(E) &=& -\frac{p}{2 (2+p) E}
                \int_{u_-}^{u_+}  \frac{u^p}{(c-u^p)}\sqrt{\frac{B(u)}{E-A(u)}} du \\
        &=& -\frac{p}{2 \sqrt{2} (2+p) E} \int_{u_-}^{u_+}  \frac{u^p}{\sqrt{E-A(u)}} du,
        \label{T-prime-turning}
\end{eqnarray}
where $u_{\pm}$ denote the turning points of the orbit $\g_E$ with $E=A(u_{\pm})$,
i.e. the intersections of the orbit $\gamma_E$ with the horizontal axis $v = 0$. Therefore,
we find that $T'(E) <0$ if $p$ is even.
Now we show that the same property also holds when $p$ is odd. Denote
\begin{align}
\label{E Integrals}
    I_1(E):= \int_{u_-}^0  \frac{u^p}{\sqrt{E-A(u)}} du, \quad
    I_2(E):= \int_0^{u_+}   \frac{u^p}{\sqrt{E-A(u)}} du,
\end{align}
then
\begin{align}
    T'(E) &= -\frac{p}{2 \sqrt{2}(2+p) E}  \big[ I_1(E) + I_2(E)\big].
\end{align}
We perform the change of variables $u=u_+ x$ and find that
\begin{eqnarray*}
    I_2(E) &=&  \int_0^{u_+}   \frac{u^p}{\sqrt{A(u_+)-A(u)}} du
           = \int_0^1 \frac{ u_+^p x^p}{\sqrt{A(u_+) - A(u_+x)}} u_+ dx\\
            &=& \sqrt{2}u_+^p \int_0^1
            \frac{ x^p}{\sqrt{c(1-x^2) - \frac{2u_+^p }{p+2} (1-x^{p+2} )}} dx.
\end{eqnarray*}
To rewrite the first integral we change variables according to $u=-|u_-|x$ and obtain
\begin{eqnarray*}
    I_1(E) &=& \int_{-|u_-|}^0  \frac{u^p}{\sqrt{A(-|u_-|)-A(u)}} du
             = \int_1^0 \frac{ -|u_-|^p x^p}{\sqrt{A(-|u_-|) - A(u_-x)}} (-|u_-|) dx\\
            &=& -\sqrt{2}|u_-|^p \int_0^1
            \frac{ x^p}{\sqrt{c(1-x^2) + \frac{2|u_-|^p }{p+2} (1-x^{p+2} )}} dx.
\end{eqnarray*}
We claim that  $|u_-| < u_+$. Indeed, we have that  $A(u)<A(-u)$ on $(0,c^{1/p})$, since
\begin{align*}
    A(u)-A(-u) = u^2 \left( \frac{c}{2} - \frac{1}{p+2} u^p \right)
                        - u^2 \left( \frac{c}{2} + \frac{1}{p+2} u^p \right)
                      =- \frac{2}{p+2} u^{p+2} <0.
\end{align*}
Moreover, $A$ is monotone on $(0,c^{1/p})$. Assuming to the contrary that $|u_-| \geq u_+$,
we would have that $A(|u_-|)\geq A(u_+)$ and hence $A(u_+)\leq A(|u_-|) < A(u_-)$, which contradicts
the fact that $A(u_+)=A(u_-)$. Hence  $0<|u_-|<u_+<c^{1/p}$, which implies that $|I_1(E)|<I_2(E)$,
and therefore, $T'(E)<0$  also in the case when $p$ is odd.
The proof of Proposition \ref{P T'<0} is complete.
\end{proof}

The following result describes the limiting points of the
energy-to-period map $E \mapsto 2T(E)$ and is proved with routine computations.

\begin{lemma}
\label{lemma-E limits}
For every $c > 0$ and $p \in \mathbb{N}$, let $E \mapsto 2T(E)$
be the mapping defined by (\ref{E period function}). Then
        \begin{equation}
          \label{limit-E-zero}
            T(0) := \lim_{E\to 0} T(E)=  \pi c^{1/2},
        \end{equation}
and there exists $T_1 \in (0,\pi)$ such that
        \begin{equation}
          \label{limit-E-one}
            T(E_c) := \lim_{E \to E_c} T(E)= T_1 c^{1/2},
        \end{equation}
with $E_c$ defined in \eqref{E-c}.
\end{lemma}

\begin{proof}
We can write (\ref{E period function}) in the explicit form
\begin{equation}
\label{T-turning-points}
T(E) = \int_{u_-}^{u_+} \frac{\sqrt{B(u)} du}{\sqrt{E - A(u)}},
\end{equation}
where the turning points $u_{\pm} \gtrless 0$ are given by the roots of $A(u_{\pm}) = E$. To prove the first assertion, we use the scaling transformation
$$
u = \left( \frac{2E}{c} \right)^{1/2} x,
$$
to rewrite the integral in (\ref{T-turning-points}) as follows:
$$
T(E) = c^{1/2} \int_{v_-}^{v_+} \frac{(1 - \mu x^p) dx}{\sqrt{1 - x^2 + 2 \mu x^{p+2}/(p+2)}}, \quad
 \mu := \frac{2^{p/2} E^{p/2}}{c^{(p+2)/2}},
$$
where $v_{\pm} \gtrless 0$ are roots of the algebraic equation
$$
v_{\pm}^2 = 1 + \frac{2}{p+2} \mu v_{\pm}^{p+2}.
$$
We note that $\mu \to 0$, $v_{\pm} \to \pm 1$ as $E \to 0$, which gives the formal limit
$$
\int_{v_-}^{v_+} \frac{(1 - \mu x^p) dx}{\sqrt{1 - x^2 + 2 \mu x^{p+2}/(p+2)}} \to \int_{-1}^1 \frac{dx}{\sqrt{1-x^2}}
= \pi \quad \mbox{\rm as} \quad \mu \to 0.
$$
This yields the limit (\ref{limit-E-zero}). The justification of the formal limit is
performed by rescaling $[v_-,v_+]$ to $[-1,1]$ and by using Lebesgue's Dominated Convergence Theorem,
since the integrand function and its limit as $\mu \to 0$ are absolutely integrable.

To prove the second assertion,  notice that for $E = E_c$, the turning points $u_{\pm}$ used in the integral (\ref{T-turning-points})
are known as $u_{\pm} = \pm c^{1/p} q_{\pm}$, where $q_+ = 1$ and $q_- > 0$ is a root of
the algebraic equation
$$
q_-^2 - \frac{2}{p+2} (-1)^p q_-^{p+2} = \frac{p}{p+2}.
$$
If $p$ is even, $q_- = 1$, while if $p$ is odd, $q_- \in (0,1)$, as follows
from the proof of Proposition \ref{P T'<0}. By splitting the integral
(\ref{T-turning-points}) into two parts we integrate over $[u_-,0]$ and $[0,u_+]$ separately and use the substitution
$u = \pm c^{1/p} x$ for the two integrals. Since $T'(E)$ is bounded for every $E > 0$ from the representation (\ref{T-prime-turning}),
including the limit $E \to E_c$, we obtain that  $T(E_c) := \lim_{E \to E_c} T(E)$ exists and is given by $T(E_c) = T_1 c^{1/2}$, where
\begin{eqnarray}
\nonumber
T_1 & := & \int_{0}^{1} \frac{(1 - x^p) dx}{\sqrt{1 - x^2 - 2(1 - x^{p+2})/(p+2)}}\\
\label{E-T1}
& \phantom{t} & + \int_{0}^{q_-} \frac{(1 - (-1)^p x^p) dx}{\sqrt{1 - x^2 - 2(1 - (-1)^p x^{p+2})/(p+2)}}.
\end{eqnarray}
Both integrals are finite and positive, from which the existence of $T_1 > 0$ is concluded. Since $T'(E) < 0$ for every $E > 0$ we have that $T_1 < \pi$.
\end{proof}

\section{Continuation of smooth periodic waves with respect to $c$}
\label{section-speed}

In Section \ref{section-monoton} we fixed the parameter $c > 0$ and
considered a continuation of the smooth periodic wave solutions $U$ with respect to
the parameter $E$ in $(0,E_c)$, where $E = 0$ corresponds to the zero solution
and $E = E_c$ corresponds to a peaked periodic wave.
The mapping $E \mapsto 2T(E)$ is found to be monotonically decreasing
according to Proposition \ref{P T'<0}. Therefore, this mapping can be inverted
for every fixed $c > 0$ and we denote the corresponding dependence by $E(T,c)$.
The range of the mapping $E \mapsto 2T(E)$, which was calculated in Lemma
\ref{lemma-E limits}, specifies the domain of the function $E(T,c)$ with respect to the parameter $T$.
The existence interval for the smooth periodic waves between the two limiting cases
(\ref{limit-E-zero}) and (\ref{limit-E-one}) obtained in Lemma~\ref{lemma-E limits} is shown in Figure~\ref{fig-region}.

\begin{figure}[h!]
    \centering
        \includegraphics[width=0.55\textwidth]{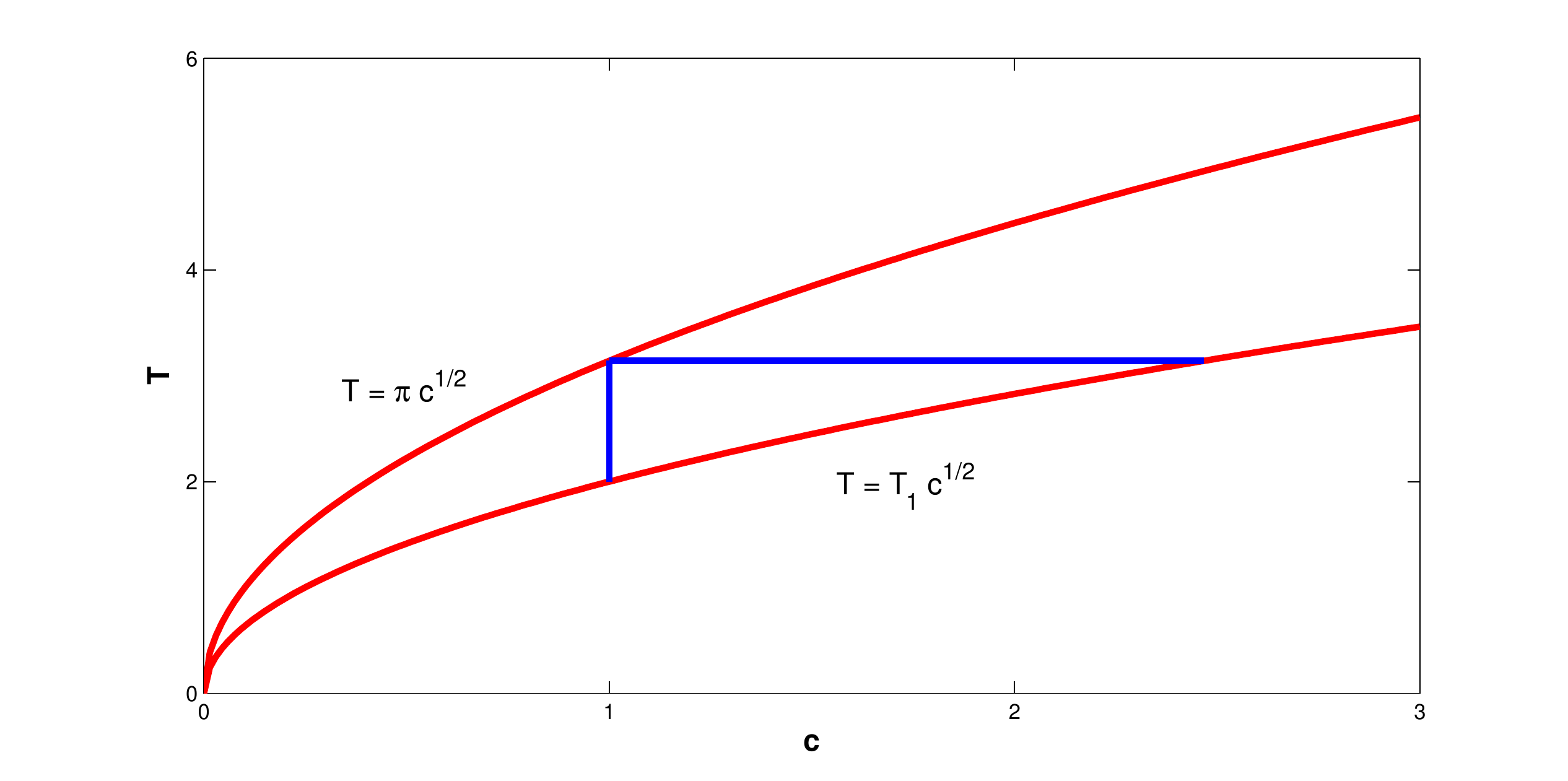}
    \caption{The existence region for smooth periodic waves in the $(T,c)$-parameter plane
    between the two limiting curves $T = \pi c^{1/2}$ and $T = T_1 c^{1/2}$ obtained in Lemma~\ref{lemma-E limits}.}
    \label{fig-region}
\end{figure}
When we fix the parameter $c > 0$, the half-period $T$ belongs to the interval
$(T_1 c^{1/2},\pi c^{1/2})$, which corresponds to the vertical line in Figure~\ref{fig-region}.
When we fix the parameter $T > 0$, the parameter $c$ belongs to the interval
$(T^2/\pi^2,T^2/T_1^2)$, which corresponds to the horizontal line in Figure~\ref{fig-region}.

In this section, we will fix the period $2T$ and consider a continuation of the smooth  periodic wave solutions $U$ with respect to the
parameter $c$ in a subset of $\mathbb{R}^+$. The next result specifies the interval of existence for the speed $c$.

\begin{lemma}
\label{lem-continuation}
For every $T > 0$ and $p \in \mathbb{N}$, there exists a family  of $2T$-periodic solutions $U=U(z;c)$ of equation~(\ref{E ODE}) parametrized by $c \in (c_0(T),c_1(T))$,
where
\begin{equation}
\label{c-0-E}
c_0(T) := \frac{T^2}{\pi^2}, \quad c_1(T) := \frac{T^2}{T_1^2} > c_0(T),
\end{equation}
with $T_1 \in (0,\pi)$ given in \eqref{E-T1} and $U \to 0$ as $c \to c_0(T)$.
Moreover,  the mapping $(c_0(T),c_1(T)) \ni c \mapsto U \in \dot{L}_{\rm per}^2(-T,T) \cap H^{\infty}_{\rm per}(-T,T)$ is $C^1$.
\end{lemma}

\begin{proof}
Notice that the scaling transformation
\begin{equation}
\label{transformation}
    U(z;c) = c^{1/p} \tilde{U}(\tilde{z}), \quad z = c^{1/2} \tilde{z}, \quad T = c^{1/2} \tilde{T},
\end{equation}
relates $2T$-periodic solutions  $U$  of the boundary-value problem (\ref{E ODE}) to  $2\tilde{T}$-periodic solution $\tilde{U}$ of the same boundary-value problem with $c$ normalized to $1$, that is,
\begin{equation}
\label{E ODE normalized}
    \frac{d}{d \tilde z} \left[ (1 - \tilde U^p) \frac{d \tilde U}{d \tilde z} \right] + \tilde U(\tilde z) = 0, \quad
    \tilde U(-\tilde T) = \tilde U(\tilde T), \quad
    \tilde U'(-\tilde T) =\tilde U'(\tilde T).
\end{equation}
Lemma \ref{L ode sys} guarantees existence of a family $\{\tilde U_E\}_{E \in (0,E_1)}$
of $2\tilde T(\tilde E)$-periodic solutions of the boundary-value problem \eqref{E ODE normalized}.
In view of Lemma \ref{lemma-E limits} and since $T$ is fixed, we have
$\tilde T(\tilde E) = c^{-1/2}T \in(T_1,\pi)$, which implies that $c$ belongs to the interval $(c_0(T),c_1(T))$,
where $c_0(T)$ and $c_1(T)$ are given by \eqref{c-0-E}. Moreover, this relation provides a one-to-one
correspondence between the parameters $c$ and $\tilde E$ in view of the fact that
$\tilde T'(\tilde E)<0$ by Proposition \ref{P T'<0} which implies that $c^{1/2} =T/\tilde T (\tilde E)$
is monotone increasing in $\tilde E$.
In view of the transformation \eqref{transformation}, we therefore obtain existence of a family
$\{U_c\}_{c \in (c_0(T),c_1(T))}$ of $2T$-periodic solutions of the boundary-value problem
\eqref{E ODE}. The value $c_0(T)$ corresponds to the zero solution,
whereas $c_1(T)$ corresponds to the peaked periodic wave.
\end{proof}

Recall that the mapping $E \mapsto 2T(E)$ can be inverted for every fixed $c > 0$, and
that the corresponding dependence is denoted by $E(T,c)$. The next result shows
that $E(T,c)$ is a monotonically increasing function of $c\in (c_0(T),c_1(T))$
for every fixed $T > 0$.

\begin{lemma}
\label{lem-increasing-in-c}
For every $T > 0$, $p \in \mathbb{N}$, the mapping $(c_0(T),c_1(T)) \ni c \mapsto E(T,c)$ is  $C^1$ and  monotonically increasing.
\end{lemma}

\begin{proof}
Using the transformation (\ref{transformation}) in the boundary-value problem \eqref{E ODE normalized}, we obtain that
$$
E(T,c) = c^{\frac{p+2}{p}} \tilde{E},
$$
where $ \tilde{E}$  is the energy level of the first integral of the second-order equation in \eqref{E ODE normalized},
$$
    \tilde{E} = \frac{1}{2} (1 - \tilde U^p)^2 \left(\frac{d U}{d \tilde{z}} \right)^2
    + \frac{1}{2} \tilde U^2 - \frac{1}{p+2} \tilde U^{p+2}.
$$
Now, as $T$ is fixed and $\tilde{T} = \tilde{T}(\tilde{E})$ is defined by (\ref{E period function})
for $c$ normalized to $1$, we can define $E(T,c)$ from the root of the following equation
\begin{equation}
\label{root-finding}
T = c^{\frac{1}{2}} \tilde{T}\left(E(T,c) c^{-\frac{p+2}{p}}\right).
\end{equation}
Since $\tilde{T}(0) = \pi$ and $\tilde{T}(E_1) = T_1$,
we have roots $E(T,c_0(T)) = 0$ and $E(T,c_1(T)) = E_c$
of the algebraic equation (\ref{root-finding}), with $E_c$ given by (\ref{E-c}) at $c = c_1(T)$.
In order to continue the roots by using the implicit function theorem for every $c \in (c_0(T),c_1(T))$,
we differentiate (\ref{root-finding}) with respect to $c$ at fixed $T$ and obtain
\begin{equation}
\label{root-finding-der}
0 = \frac{1}{2} c^{-\frac{1}{2}} \tilde{T}(\tilde{E}) - \frac{p+2}{p} E
c^{-\frac{3p+4}{2p}} \tilde{T}'(\tilde{E}) + c^{-\frac{p+4}{2 p}} \tilde{T}'(\tilde{E}) \frac{\partial E(T,c)}{\partial c}.
\end{equation}
By Proposition \ref{P T'<0}, we have $\tilde{T}'(\tilde{E}) < 0$ for $\tilde{E} \in (0,E_1)$,
so that we can rewrite (\ref{root-finding-der}) as follows:
\begin{equation}
\label{root-finding-der-pos}
\left| \tilde{T}'(\tilde{E}) \right| \frac{\partial E(T,c)}{\partial c}
= \frac{1}{2} c^{\frac{2}{p}} \tilde{T}(\tilde{E}) + \frac{p+2}{p} E
c^{-1} \left| \tilde{T}'(\tilde{E}) \right| > 0.
\end{equation}
Recall that $\tilde{T}'(\tilde{E})$ is bounded and nonzero for every $E \in (0,E_1)$ and in the limit $E \to E_1$.
By the implicit function theorem and thanks to the smoothness of all dependencies,
there exists a unique, monotonically increasing $C^1$ map $(c_0(T),c_1(T)) \ni c \mapsto E(T,c)$ such that
$E(T,c)$ is a root of equation (\ref{root-finding}) and $E(T,c_1(T)) = E_c$, where
$E_c$ is given by (\ref{E-c}) at $c = c_1(T)$.
\end{proof}

We shall now consider how the $ L^2_{\rm per}(-T,T)$ norm of the periodic wave $U$ with fixed period $2T$
depends on the parameter $c$. In order to prove that it is an increasing function
of $c$ in $(c_0(T),c_1(T))$, we obtain a number of identities satisfied by the periodic wave $U$.
This result will be used in the proof of Proposition~\ref{positive-operator} in Section \ref{section-Krein}.

\begin{lemma}
\label{lem-increasing}
For every $T > 0$, $p \in \mathbb{N}$, the mapping $(c_0(T),c_1(T)) \ni c \mapsto \| U \|_{L^2_{\rm per}(-T,T)}^2$ is $C^1$
and monotonically increasing. Moreover, if the operator $L$ is defined by \eqref{operator-L}, then
$\partial_c U \in \dot{L}_{\rm per}^2(-T,T)$ satisfies
\begin{equation}
\label{inhom-partial-c}
L \partial_c U = -U
\end{equation}
and
\begin{equation}
\label{slope-condition}
    \langle\partial_c U,U\rangle_{L^2_{\rm per}}  > 0.
\end{equation}
\end{lemma}

\begin{proof}
Integrating (\ref{E ODE}) in $z$ with zero mean, we can write
\begin{equation}
\label{equation-0}
(c-U^p) \partial_z U + \partial_z^{-1} U = 0.
\end{equation}
From here, multiplication by $\partial_z^{-1} U$ and integration by parts yield
\begin{equation}
\label{equation-1}
\| \partial_z^{-1} U \|^2_{L^2_{\rm per}(-T,T)} = c \| U \|_{L^2_{\rm per}(-T,T)}^2 - \frac{1}{p+1} \int_{-T}^T U^{p+2} dz.
\end{equation}
On the other hand, integrating (\ref{first-order}) over the period $2T$ and using equations
(\ref{equation-0}) and (\ref{equation-1}) yield
\begin{eqnarray}
\nonumber
2 E(T,c) T & = & \frac{c}{2} \| U \|_{L^2_{\rm per}(-T,T)}^2 - \frac{1}{p+2} \int_{-T}^T U^{p+2} dz + \frac{1}{2}
\left\| (c-U^p) \frac{d U}{d z} \right\|^2_{L^2_{\rm per}(-T,T)} \\
\nonumber
& = & \frac{c}{2} \| U \|_{L^2_{\rm per}(-T,T)}^2 - \frac{1}{p+2} \int_{-T}^T U^{p+2} dz + \frac{1}{2}
\| \partial_z^{-1} U \|^2_{L^2_{\rm per}(-T,T)} \\
& = & c \| U \|_{L^2_{\rm per}(-T,T)}^2 - \frac{(3p+4)}{2(p+1)(p+2)} \int_{-T}^T U^{p+2} dz.
\label{equation-2}
\end{eqnarray}
Expressing $c  \| U \|_{L^2_{\rm per}(-T,T)}^2$ from equations (\ref{equation-1}) and (\ref{equation-2}), we obtain
\begin{eqnarray}
\| \partial_z^{-1} U \|^2_{L^2_{\rm per}} = 2 E(T,c) T + \frac{p}{2(p+1)(p+2)} \int_{-T}^T U^{p+2} dz.
\label{equation-3}
\end{eqnarray}
From the fact that $U$ is a critical point of $H(u) + c Q(u)$ given by (\ref{energy}) and (\ref{momentum})
for a fixed period $2T$, we obtain
\begin{equation}
\label{variational-1}
\frac{d \mathcal{H}}{dc} + c \frac{d \mathcal{Q}}{dc} = 0,
\end{equation}
where
\begin{eqnarray}
\nonumber
\mathcal{H}(c) & = & - \frac{1}{2} \| \partial_z^{-1} U \|^2_{L^2_{\rm per}(-T,T)} - \frac{1}{(p+1)(p+2)}  \int_{-T}^T U^{p+2} dz \\
& = & -E(T,c) T - \frac{(p+4)}{4(p+1)(p+2)}\int_{-T}^T U^{p+2} dz
\label{variational-2}
\end{eqnarray}
and
\begin{eqnarray}
\nonumber
c \mathcal{Q}(c) & = & \frac{c}{2} \| U \|^2_{L^2_{\rm per}(-T,T)} \\
& = & E(T,c) T + \frac{(3p+4)}{4(p+1)(p+2)} \int_{-T}^T U^{p+2} dz
\label{variational-3}
\end{eqnarray}
are simplified with the help of equations (\ref{equation-2}) and (\ref{equation-3}) again.
Next, we differentiate (\ref{variational-2}) and (\ref{variational-3}) in $c$ for fixed $T$
and use (\ref{variational-1}) to obtain a constraint
\begin{eqnarray}
\nonumber
\frac{d \mathcal{H}}{dc} + c \frac{d \mathcal{Q}}{dc} & = & - \frac{(p+4)}{4(p+1)(p+2)} \frac{d}{dc} \int_{-T}^T U^{p+2} dz - \mathcal{Q}(c)
+ \frac{(3p+4)}{4(p+1)(p+2)} \frac{d}{dc} \int_{-T}^T U^{p+2} dz \\
& = & - \mathcal{Q}(c) + \frac{p}{2(p+1)(p+2)} \frac{d}{dc} \int_{-T}^T U^{p+2} dz = 0.
\label{variational-4}
\end{eqnarray}
From (\ref{root-finding-der-pos}), (\ref{variational-1}), (\ref{variational-2}), and (\ref{variational-4}), we finally obtain
\begin{eqnarray}
\nonumber
c \frac{d \mathcal{Q}}{dc} = -\frac{d \mathcal{H}}{dc} & = & T \frac{\partial E(T,c)}{\partial c} +
\frac{(p+4)}{4(p+1)(p+2)} \frac{d}{dc} \int_{-T}^T U^{p+2} dz \\
\label{positivity-derivative}
& = & T \frac{\partial E(T,c)}{\partial c} +  \frac{p+4}{2p} \mathcal{Q}(c) > 0.
\end{eqnarray}

To prove the second assertion, recall that the family of periodic waves $U(z;c)$ is $C^1$ with respect to $c$
by Lemma~\ref{lem-continuation}. Differentiating the second-order equation in \eqref{E ODE}
with respect to $c$ at fixed period $2T$ and integrating it twice with zero mean yields equation
(\ref{inhom-partial-c}). Notice that $\partial_c U$ is again $2T$-periodic,
since the period of $U$ is fixed independently of  $c$. Finally, we find that
\begin{equation*}
\langle \partial_c U, U \rangle_{L^2_{\rm per}} = \frac{1}{2} \frac{d}{dc} \| U \|_{L^2_{\rm per}}^2 >0,
\end{equation*}
since by the first assertion, the mapping $c\mapsto \| U \|_{L^2_{\rm per}}^2$ is monotonically increasing.
\end{proof}

As an immediate consequence of Lemmas \ref{lem-continuation} and \ref{lem-increasing},
we  prove the following result which will be used in the proof of  Proposition \ref{negative-index-L}
in Section \ref{section-negative}.

\begin{corollary}
\label{cor-increasing}
For every $T > 0$, $p \in \mathbb{N}$ and $c\in (c_0(T),c_1(T))$,
the periodic solution $U$ of the boundary-value problem  \eqref{E ODE}  satisfies
\begin{equation}
\label{positivity}
\int_{-T}^T U^{p+2} dz > 0.
\end{equation}
\end{corollary}

\begin{proof}
It follows from (\ref{variational-4}) that
\begin{equation}
\label{positivity-what-we-need}
\frac{d}{dc} \int_{-T}^T U^{p+2} dz = \frac{2(p+1)(p+2)}{p} \mathcal{Q}(c) > 0, \quad c \in (c_0(T),c_1(T)).
\end{equation}
On the other hand, $\int_{-T}^T U^{p+2} dz = 0$ at $c = c_0(T)$  by Lemma \ref{lem-continuation}. Integrating the inequality (\ref{positivity-what-we-need}) for $c > c_0(T)$ implies positivity
of $\int_{-T}^T U^{p+2} dz$ .
\end{proof}

\section{Negative index of the operator $L$}
\label{section-negative}

Recall that $T(E) \to T(0) = \pi c^{1/2}$ and $U\to 0$ as $E \to 0$ in view of Lemma~\ref{lemma-E limits}. In this limit, the operator given by (\ref{operator-L}) becomes an integral operator with constant coefficients,
\begin{equation*}
    L_0 = P_0( \partial_z^{-2} + c)P_0: \dot L_{per}^2(-T(0),T(0)) \to \dot L_{per}^2(-T(0),T(0)),
\end{equation*}
whose spectrum can be computed explicitly as
\begin{equation}
\label{spectrum-zero-amplitude}
    \sigma(L_0) = \left\{ c(1 - n^{-2}), \quad n \in \mathbb{Z} \backslash \{0\} \right\},
\end{equation}
by using Fourier series.  For every $c > 0$, the spectrum of $L_0$ is purely discrete and consists of double eigenvalues accumulating
to the point $c$. All double eigenvalues are strictly positive except for the lowest eigenvalue,
which is located at the origin. As is shown in \cite{JP} with a perturbation argument for $p = 1$ and $p = 2$, the spectrum of $L$
for $E$ near $0$ includes a simple negative eigenvalue,
a simple zero eigenvalue, and the positive spectrum is bounded away from zero.
We will show that this conclusion remains true for the entire family of smooth periodic waves.
Let us first prove the following.

\begin{lemma}
\label{lem-L}
For every $c > 0$, $p \in \mathbb{N}$, and $E \in (0,E_c)$, the operator $L$ given by (\ref{operator-L}) is self-adjoint
and its spectrum includes a countable set of isolated eigenvalues below
\begin{equation}
\label{bound C-}
C_-(E) := \inf_{z \in[-T(E),T(E)]} (c - U(z)^p) > 0.
\end{equation}
\end{lemma}

\begin{proof}
The self-adjoint properties of $L$ are obvious. For every $E \in (0,E_c)$, there are positive constants $C_{\pm}(E)$
such that
\begin{equation}
\label{bounds-U}
C_-(E) \leq c - U(z)^p \leq C_+(E) \quad \mbox{\rm for every} \;\; z \in [-T(E),T(E)].
\end{equation}
For the rest of the proof we use the short notation  $T = T(E)$. The eigenvalue equation $(L-\lambda I) v = 0$ for $v \in \dot{L}^2_{\rm per}(-T,T)$ is equivalent to the
spectral problem
\begin{equation}
\label{tilde-L}
P_0 (c-U^p-\lambda) P_0 v = - P_0 \partial_z^{-2} P_0 v.
\end{equation}
Under the condition $\lambda < C_-(E)$, we have $c - U^p - \lambda \geq C_-(E) - \lambda > 0$.
Setting
\begin{equation}
\label{variable-w}
w := (c-U^p-\lambda)^{1/2} P_0 v \in L^2_{\rm per}(-T,T), \quad \lambda < C_-(E),
\end{equation}
we find that $\lambda$ is an eigenvalue of the spectral problem (\ref{tilde-L})
if and only if $1$ is an eigenvalue of the self-adjoint operator
\begin{equation}
\label{K-operator}
K(\lambda) = - (c-U^p-\lambda)^{-1/2} P_0 \partial_z^{-2} P_0 (c-U^p-\lambda)^{-1/2} : L^2_{\rm per}(-T,T) \to L^2_{\rm per}(-T,T),
\end{equation}
that is\footnote{This reformulation can be viewed as an adjoint version
of the Birmann--Schwinger principle used in analysis of isolated eigenvalues
of Schr\"{o}dinger operators with rapidly decaying potentials \cite{Gustaf}.}, $w = K(\lambda) w$.
The operator $K(\lambda)$ for every $\lambda < C_-(E)$ is a compact (Hilbert--Schmidt)
operator thanks to the bounds (\ref{bounds-U}) and the compactness of $P_0 \partial_z^{-2} P_0$.
Consequently, the spectrum of $K(\lambda)$ in $L^2_{\rm per}(-T,T)$
for every $\lambda < C_-(E)$ is purely discrete and consists of isolated eigenvalues.
Moreover, these eigenvalues are positive thanks to the positivity of $K(\lambda)$, as follows:
\begin{equation}
\label{representation-K}
\langle K(\lambda) w, w \rangle_{L^2_{\rm per}} =  \| P_0 \partial_z^{-1} P_0 (c-U^p-\lambda)^{-1/2} w \|_{L^2_{\rm per}}^2 \geq 0, \quad
\forall w \in L^2_{\rm per}(-T,T).
\end{equation}
We note that
\begin{itemize}
\item[(a)] $K(\lambda) \to 0^+$ as $\lambda \to -\infty$,
\item[(b)] $K'(\lambda) > 0$ for every $\lambda < C_-(E)$.
\end{itemize}
Claim (a) follows from (\ref{representation-K}) via spectral calculus:
\begin{eqnarray*}
\langle K(\lambda) w, w \rangle_{L^2_{\rm per}} \sim |\lambda|^{-1}
\| P_0 \partial_z^{-1} P_0 w \|_{L^2}^2 \quad \mbox{\rm as} \quad \lambda \to -\infty.
\end{eqnarray*}
Claim (b) follows from the differentiation of $K(\lambda)$,
\begin{eqnarray*}
\langle K'(\lambda) w, w \rangle_{L^2_{\rm per}} & = &
\frac{1}{2} \langle \rho(\lambda)K(\lambda)w,w\rangle_{L^2_{\rm per}}
+\frac{1}{2} \langle K(\lambda) \rho(\lambda)w,w\rangle_{L^2_{\rm per}},
\end{eqnarray*}
where we have defined the weight function $\rho(\lambda) := (c - U^p - \lambda)^{-1}$
which is strictly positive and uniformly bounded thanks to \eqref{bounds-U}.
Since $K(\lambda)$ is positive due to (\ref{representation-K}), both terms
in the above expression are positive in view of a generalization of Sylvester's law
of inertia for differential operators, see Theorem~4.2~in~\cite{Pelinovsky2011}.
Indeed, to prove that the first term is positive it suffices to show that the
eigenvalues $\mu$ of $\rho(\lambda) K(\lambda) $ are positive. The corresponding
spectral problem $\rho(\lambda) K(\lambda) w = \mu  w$ is equivalent to
$\rho(\lambda)^{1/2} K(\lambda) \rho(\lambda)^{1/2} v = \mu  v$ in view of
the substitution $w=\rho(\lambda)^{1/2}v$. By Sylvester's law, the number of negative
eigenvalues of $K(\lambda)$ is equal to the number of negative eigenvalues of
the congruent operator $\tilde K(\lambda) = \rho(\lambda)^{1/2} K(\lambda) \rho(\lambda)^{1/2}$.
Therefore, $\rho(\lambda)K(\lambda)$ is positive  in view of the positivity of  $K(\lambda)$.
The second term can be treated in the same way.

It follows from claims (a) and (b) that positive isolated eigenvalues of $K(\lambda)$
are monotonically increasing functions of $\lambda$ from the zero level as $\lambda \to -\infty$.
The location and number of crossings of these eigenvalues with the unit level gives the location and
number of eigenvalues $\lambda$ in the spectral problem (\ref{tilde-L}). The compactness of $K(\lambda)$
for $\lambda < C_-(E)$ therefore implies that there exists a countable (finite or infinite)
set of isolated eigenvalues of $L$ below $C_-(E)$.
\end{proof}

Next, we inspect analytical properties of eigenvectors for isolated eigenvalues below $C_-(E) > 0$ given by (\ref{bound C-}).

\begin{lemma}
\label{lem-Eigenvector}
Under the condition of Lemma \ref{lem-L}, let $\lambda_0 < C_-(E)$ be an eigenvalue of the operator
$L$ given by (\ref{operator-L}). Then, $\lambda_0$ is at most double and
the eigenvector $v_0$ belongs to $\dot{L}^2_{\rm per}(-T(E),T(E)) \cap H^{\infty}_{\rm per}(-T(E),T(E))$.
\end{lemma}

\begin{proof}
As in the proof of the previous Lemma, we use the shorthand  $T = T(E)$ for lighter notation. The eigenvector $v_0 \in \dot{L}^2_{\rm per}(-T,T)$ for the eigenvalue $\lambda_0 < C_-(E)$
satisfies the spectral problem (\ref{tilde-L}) written as the integral equation
\begin{equation}
\label{integral-eq}
P_0 \partial_z^{-2}P_0 v_0 + P_0 (c-U^p-\lambda_0)P_0 v_0 = 0.
\end{equation}
Since $U \in H^{\infty}_{\rm per}(-T,T)$ and $c - U^p - \lambda_0 \geq C_-(E) - \lambda_0 > 0$,
we obtain that $v_0\in H^{2}_{\rm per}(-T,T)$, and by bootstrapping arguments we find that  $v_0 \in H^{\infty}_{\rm per}(-T,T)$. Applying two derivatives to the integral equation
(\ref{integral-eq}), we obtain the equivalent differential equation for
the eigenvector $v_0 \in \dot{L}^2_{\rm per}(-T,T) \cap H^{\infty}_{\rm per}(-T,T)$
and the eigenvalue $\lambda_0 < C_-(E)$:
\begin{equation}
\label{differential-eq}
v_0 +\partial_z^2 \left[ (c-U^p - \lambda_0) v_0 \right] = 0.
\end{equation}
The second-order differential equation (\ref{differential-eq}) admits at most two linearly independent solutions
in $\dot{L}^2_{\rm per}(-T,T)$ and so does the integral equation (\ref{integral-eq}) for an eigenvalue $\lambda_0 < C_-(E)$.
Since $L$ is self-adjoint, the eigenvalue $\lambda_0$ is not defective\footnote{Recall that
the eigenvalue is called defective if its algebraic multiplicity exceeds its geometric multiplicity.},
and hence the multiplicity of $\lambda_0$ is at most two.
\end{proof}

We are now ready to prove the main result of this section. This proves part (b) of Theorem~\ref{theorem-main}.

\begin{proposition}
\label{negative-index-L}
For every $c > 0$, $p \in \mathbb{N}$, and $E \in (0,E_c)$, the operator $L$
given by (\ref{operator-L}) has exactly one simple negative eigenvalue, a simple zero eigenvalue,
and the rest of the spectrum is positive and bounded away from zero.
\end{proposition}

\begin{proof}
Thanks to Lemma \ref{lem-L}, we only need to inspect the multiplicity of negative and zero eigenvalues
of $L$. By Lemma \ref{lem-Eigenvector}, the zero eigenvalue $\lambda_0 = 0 < C_-(E)$ can be at most double.
The first eigenvector $v_0 = \partial_z U \in \dot{L}^2_{\rm per}(-T(E),T(E)) \cap H^{\infty}_{\rm per}(-T(E),T(E))$ for $\lambda_0 = 0$ follows by the translational
symmetry. Indeed, differentiating (\ref{E ODE}) with respect to $z$, we verify that $v_0$ satisfies
the differential equation (\ref{differential-eq}) with $\lambda_0 = 0$ and, equivalently,
the integral equation (\ref{integral-eq}) with $\lambda_0 = 0$.

Another linearly independent solution $v_1 = \partial_E U$ of the same equation
(\ref{differential-eq}) with $\lambda_0 = 0$ is obtained by differentiating (\ref{E ODE}) with respect to $E$. Here we understand
the family $U(z;E)$ of smooth $2 T(E)$-periodic solutions constructed in Lemma~\ref{L ode sys},
where the period $2 T(E)$ is given by (\ref{E period function}) and is a smooth function of $E$.
Now, we show that the second solution $v_1$ is not $2 T(E)$-periodic under the condition $T'(E) < 0$
established in Proposition~\ref{P T'<0}. Consequently, the zero eigenvalue $\lambda_0 = 0$ is simple.
For simplicity, we assume that the family $U(z;E)$ satisfies the condition
\begin{equation}
\label{norming-condition}
U(\pm T(E);E) = 0
\end{equation}
at the end points, which can be fixed by translational symmetry.
By differentiating the first boundary condition in (\ref{E ODE}) with respect to $E$,
we obtain
\begin{eqnarray*}
\partial_E U(-T(E);E) - T'(E) \partial_z U(-T(E);E) = \partial_E U(T(E);E) + T'(E) \partial_z U(T(E);E).
\end{eqnarray*}
Notice that $\partial_z U(\pm T(E);E) \neq 0$, since otherwise the periodic solution $U$ would be identically zero in view of  (\ref{norming-condition}) which is only possible for $E = 0$. Since $T'(E) \neq 0$ by Proposition \ref{P T'<0},
the solution $v_1 = \partial_E U$ is not $2T(E)$-periodic and therefore the zero eigenvalue $\lambda_0 = 0$ is simple for the entire family of smooth $T(E)$-periodic solutions.

Next, we show that the spectrum of $L$ includes at least one negative eigenvalue. Indeed,
from the integral version of the differential equation (\ref{E ODE}),
$$
P_0 \left(c - \frac{1}{p+1} U^p \right)P_0 U + P_0 \partial_z^{-2} P_0 U = 0,
$$
we obtain that $L U = -\frac{p}{p+1} P_0 U^{p+1}$, which implies that
\begin{equation}
\langle L U, U \rangle_{L^2_{\rm per}} = -\frac{p}{p+1} \int_{-T(E)}^{T(E)} U^{p+2} dz < 0.
\label{negativity}
\end{equation}
The last inequality is obvious for even $p$. For odd $p$ it follows from Corollary \ref{cor-increasing} for given $T(E) \in (T_1 c^{1/2},\pi c^{1/2})$ fixed.
In both cases, we have shown that $L$ has at least one negative eigenvalue for every $E \in (0,E_c)$.

Finally, the spectrum of $L$ includes at most one simple negative eigenvalue. Indeed,
the family of $2 T(E)$-periodic solutions is smooth with respect to the  parameter $E \in (0,E_c)$
and it reduces to the zero solution as $E\to 0$. It follows from the spectrum (\ref{spectrum-zero-amplitude}) for the operator $L_0$ at the zero solution, and the preservation of the simple zero eigenvalue with the eigenvector
$\partial_z U$ for every $E \in (0,E_c)$, that the splitting of a double zero eigenvalue for $E \neq 0$
results in appearance of at most one negative eigenvalue of $L$.
Thus, there exists exactly one simple negative eigenvalue of $L$ for every $E \in (0,E_c)$.
\end{proof}

\section{Applications of the Hamilton--Krein theorem}
\label{section-Krein}

Since $L$ has a simple zero eigenvalue in $\dot{L}^2_{\rm per}(-T,T)$ by Proposition \ref{negative-index-L}
with the eigenvector $v_0 = \partial_z U$,
eigenvectors $v \in \dot{H}^1_{\rm per}(-T,T)$ of the spectral problem $\lambda v = \partial_z Lv$
for nonzero eigenvalues $\lambda$ satisfy the constraint $\langle U, v \rangle_{L^2_{\rm per}} = 0$,
see definition  (\ref{constraint})  of the space $L^2_c$.
Since $\partial_z$ is invertible in space $\dot{L}^2_{\rm per}(-T,T)$ and the inverse operator is bounded
from $\dot{L}^2_{\rm per}(-T,T)$ to itself, we can rewrite the spectral
problem $\lambda v = \partial_z Lv$  in the equivalent form
\begin{equation}
\label{eigenvalue-problem}
\lambda P_0 \partial_z^{-1} P_0 v = L v, \quad v \in \dot{L}^2_{\rm per}(-T,T).
\end{equation}
In this form, the Hamilton--Krein theorem from \cite{Haragus} applies directly in $L^2_c$. According to this theorem,
the number of unstable eigenvalues with $\lambda \notin i \mathbb{R}$ is bounded by the number of negative eigenvalues
of $L$ in the constrained space $L^2_c$. Therefore, we only need to show that the operator $L$ is positive in $L^2_c$
with only a simple zero eigenvalue due to the translational invariance in order to prove part (c) of Theorem \ref{theorem-main}.
The corresponding result is given by the following proposition.

\begin{proposition}
For every $c > 0$, $p \in \mathbb{N}$, and $E \in (0,E_c)$, the operator $L |_{L^2_c} : L^2_c \to L^2_c$,
where $L$ is given by (\ref{operator-L}), has a simple zero eigenvalue
and a positive spectrum bounded away from zero.
\label{positive-operator}
\end{proposition}

\begin{proof}
The proof relies on a well-known criterion (see for example Lemma~1 in \cite{Hakkaev2} or Theorem~4.1~in~\cite{Pelinovsky2011})
which ensures positivity of the self-adjoint operator $L$ with properties obtained in Proposition~\ref{negative-index-L},
when it is restricted to a co-dimension one subspace. Positivity of $L |_{L^2_c} : L^2_c \to L^2_c$
is achieved under the condition
\begin{equation}
\label{to-be-shown}
\langle L^{-1} U,U\rangle_{L^2_{\rm per}} <0.
\end{equation}
To show (\ref{to-be-shown}), we observe that  ${\rm Ker}(L) = {\rm span}\{ v_0\}$, where $v_0=\partial_z U$
and $\langle U, v_0 \rangle_{L^2_{\rm per}}  = 0$ implies that $U \in {\rm Ker}(L)^{\perp}$.
By Fredholm's Alternative (see e.g.~Theorem~B.4~in~\cite{Pelinovsky2011}), $L^{-1} U$ exists in
$\dot{L}_{\rm per}^2(-T,T)$ and can be made unique by the orthogonality condition $\langle L^{-1} U, v_0 \rangle_{L^2_{\rm per}} = 0$.
By Lemma~\ref{lem-increasing}, we have the existence of $\partial_c U \in \dot{L}_{\rm per}^2(-T,T)$ such that
$L \partial_c U =-U$, see equation (\ref{inhom-partial-c}). Moreover,
$\langle \partial_c U, v_0 \rangle_{L^2_{\rm per}} =0$, since $\partial_c U$ and $v_0=\partial_z U$ have opposite parity.
Therefore, $\partial_c U = L^{-1} U$ and we obtain
$$
\langle L^{-1} U,U\rangle_{L^2_{\rm per}}  = - \langle \partial_c U,U \rangle_{L^2_{\rm per}}  <0,
$$
where the strict negativity follows from Lemma~\ref{lem-increasing}.
\end{proof}

The proof of Theorem \ref{theorem-main} follows from the results of
Propositions \ref{P T'<0}, \ref{negative-index-L}, and \ref{positive-operator}.

\section*{Acknowledgements}
A.G.~is supported by the Austrian Science Fund (FWF) project J3452 ``Dynamical
Systems Methods in Hydrodynamics''. The work of D.P.~is supported by the Ministry of Education
and Science of Russian Federation (the base part of the State task No.~2014/133, project No.~2839).
The authors thank Todd Kapitula (Calvin College) for pointing out an error in an early version of this manuscript.


\begin{thebibliography}{99}

\bibitem{Bronski} J.C. Bronski, M.A. Johnson, and T. Kaputula, ``An index theorem for the stability of
periodic traveling waves of KdV Type", Proc. Royal Soc. Edinburgh A {\bf 141} (2011), 1141--1173.

\bibitem{Chicone}
C.~Chicone.
\newblock {\em {Ordinary Differential Equations with Applications}}.
\newblock Springer, New York, 2006.

\bibitem{JG} E.R. Johnson and R.H.J. Grimshaw, ``The modified reduced Ostrovsky equation:
integrability and breaking'', Phys. Rev. E {\bf 88} (2013), 021201(R) (5 pages).

\bibitem{JP} E.R. Johnson and D.E. Pelinovsky, ``Orbital stability of periodic waves in the class of reduced Ostrovsky equations",
{\bf 261} (2016), 3268--3304.

\bibitem{Johnson} M.A. Johnson, ``Nonlinear stability of periodic traveling wave solutions of the generalized
Korteweg-de Vries equation", SIAM J. Math. Anal. {\bf 41} (2009), 1921--1947.

\bibitem{Vill1} A. Garijo and J. Villadelprat, ``Algebraic and analytical tools for the study of the period function",
J. Diff. Eqs. {\bf 257} (2014), 2464--2484.

\bibitem{Vill2} A. Geyer and J. Villadelprat, ``On the wave length of smooth periodic traveling waves
of the Camassa--Holm equation", J. Diff. Eqs. {\bf 259} (2015), 2317--2332.

\bibitem{Grau2011} M.~Grau, F.~Ma{\~{n}}osas, and J.~Villadelprat, ``A Chebyshev criterion for Abelian integrals",
Trans. Amer. Math. Soc. {\bf 363} (2011), 109--129.

\bibitem{GH} R.H.J. Grimshaw, K. Helfrich, and
E.R. Johnson, ``The reduced Ostrovsky equation: integrability and breaking'',
Stud. Appl. Math. {\bf 129} (2012), 414--436.

\bibitem{Grimshaw} R.H.J. Grimshaw, L.A. Ostrovsky, V.I. Shrira, and Yu.A. Stepanyants,
``Long nonlinear surface and internal gravity waves in a rotating ocean'',
Surv. Geophys. {\bf 19} (1998), 289--338.

\bibitem{GP} R. Grimshaw and D.E. Pelinovsky,  ``Global existence of small-norm solutions
in the reduced Ostrovsky equation",  DCDS A {\bf 34} (2014), 557--566.

\bibitem{Gustaf} S.J. Gustafson and I.M. Sigal, {\em Mathematical Concepts of Quantum Mechanics}
(Springer--Verlag, Berlin Heidelberg, 2006).

\bibitem{Hakkaev1} S. Hakkaev, M. Stanislavova, and A. Stefanov, ``Periodic travelling waves
of the short pulse equation: existence and stability", preprint (2015)

\bibitem{Hakkaev2} S. Hakkaev, M. Stanislavova, and A. Stefanov, ``Spectral stability for classical
periodic waves of the Ostrovsky and short pulse models", arXiv: 1604.03024 (2016)

\bibitem{Haragus} M. Haragus and T. Kapitula, ``On the spectra of periodic waves for infinite-dimensional
Hamiltonian systems", Physica D {\bf 237}  (2008),  2649--2671.

\bibitem{Ilyashenko2006} Y. Ilyashenko and S. Yakovlenko, {\em Lectures on analytic differential equations},
Graduate Studies in Mathematics {\bf 86} (AMS, Providence, 2007).

\bibitem{Kollar2014} R. Koll\'{a}r and P.D. Miller, ``Graphical Krein Signature Theory and Evans--Krein Functions", SIAM Rev. {\bf 56} (2014), 73–-123.

\bibitem{LPS1} Y. Liu, D. Pelinovsky, and A. Sakovich,``Wave breaking in the Ostrovsky--Hunter equation",
SIAM J. Math. Anal. {\bf 42} (2010), 1967--1985.

\bibitem{LPS2} Y. Liu, D. Pelinovsky, and A. Sakovich,``Wave breaking in the short-pulse equation",
Dynamics of PDE {\bf 6} (2009), 291--310.

\bibitem{Nataly1} F. Natali and A. Neves,``Orbital stability of periodic waves",
IMA J. Appl. Math.  {\bf 79}  (2014),  1161--1179.

\bibitem{Nataly2} F.  Natali and A. Pastor, ``Orbital stability of periodic waves
for the Klein-Gordon-Schr\"{o}dinger system", Discrete Contin. Dyn. Syst. {\bf 31}  (2011),
221--238.

\bibitem{Ostrov} L.A. Ostrovsky, ``Nonlinear internal waves
in a rotating ocean'', Okeanologia {\bf 18} (1978), 181--191.

\bibitem{Pelinovsky2011} D.E. Pelinovsky,  {\em Localization in Periodic Potentials: From Schr\"{o}dinger Operators to the Gross–Pitaevskii Equation},
 London Mathematical Society, Lecture Note Series {\bf 390} (Cambridge University Press, 2011).

\bibitem{Pel} D.E. Pelinovsky,  ``Spectral stability of nonlinear waves in KdV-type evolution equations",
{\em Nonlinear Physical Systems: Spectral Analysis, Stability, and Bifurcations}
(Edited by O.N. Kirillov and D.E. Pelinovsky) (Wiley-ISTE, NJ, 2014), 377--400.


\bibitem{PelSak} D. Pelinovsky and A. Sakovich, ``Global well-posedness of the short-pulse and
    sine--Gordon equations in energy space", Comm. Part. Diff. Eqs. {\bf 35} (2010), 613--629.

\bibitem{SSK10} A. Stefanov, Y. Shen, and P.G. Kevrekidis, ``Well-posedness and small
data scattering for the generalized Ostrovsky equation", J. Diff. Eqs. {\bf 249} (2010), 2600--2617.

\bibitem{Stefanov} M. Stanislavova and A. Stefanov, ``On the spectral problem $L u = \lambda u'$
and applications", Commun. Math. Phys. {\bf 343} (2016), 361--391.


\end{thebibliography}
\end{document}